\documentclass[12pt,a4paper]{article}
\usepackage[T1]{fontenc}

\usepackage{amsmath,amsthm,amsfonts,amssymb}
\usepackage{enumerate}
\usepackage[bookmarks=false,colorlinks=true,linkcolor=black,citecolor=black,filecolor=black,urlcolor=black]{hyperref}
\usepackage{cleveref}
\usepackage{multirow}
\usepackage{booktabs}
\usepackage[table]{xcolor}

\newtheorem{theorem}{Theorem}[section]
\newtheorem{definition}[theorem]{Definition}
\newtheorem{corollary}[theorem]{Corollary}
\newtheorem{proposition}[theorem]{Proposition}
\newtheorem{lemma}[theorem]{Lemma}

\newtheorem{remark}[theorem]{Remark}

\newtheorem{example}[theorem]{Example}

\newcommand{\Aut}{\operatorname{Aut}}

\newcommand{\rank}{\operatorname{rank}}

\usepackage{url}

\title{Generalized Hadamard full propelinear codes%\footnote{Submitted to \emph{Designs, Codes and Cryptography}.}
}

\newcommand{\abk}{\allowbreak}

\newcommand{\Sym}{\operatorname{Sym}}

\newcommand{\GH}{\operatorname{GH}}

\newcommand{\GHFP}{\operatorname{GHFP}}

\author{Jos\'e Andr\'es Armario\footnote{\url{E-mail: armario@us.es}}\\
\small Universidad de Sevilla, Sevilla, Spain\\
Ivan Bailera\footnote{\url{E-mail: ivan.bailera@uab.cat}}\\
        \small Universitat Aut\`onoma de Barcelona, Bellaterra, Spain\\
Ronan Egan\footnote{\url{E-mail: ronan.egan@nuigalway.ie}}\\
        \small National University of Ireland Galway, Galway, Ireland
}

\begin{document}

\maketitle

\begin{abstract}
Codes from generalized Hadamard matrices have already been introduced. Here we deal with  these codes when the generalized Hadamard matrices are cocyclic. As a consequence, a new class of codes that we call {\it generalized Hadamard full propelinear codes} turns out. We prove that their existence is equivalent to the existence of central relative $(v,w,v,v/w)$-difference sets. Moreover, some structural properties of these codes are studied and  examples  are provided.\\


 \end{abstract}

\newpage

\section{Introduction}
Let $G$ and $U$ be finite groups, with $U$ abelian, of orders $v$ and $w$, respectively. A map 
$\psi: G\times G\rightarrow U$ such that
\begin{equation}
\label{CocycleIdentity}
\psi(g,h)\psi(gh,k)=\psi(g,hk)\psi(h,k) \quad\forall \, g,h,k\in G
\end{equation}
is a \emph{cocycle} (\emph{over $G$}, \emph{with coefficients 
in $U$}). 
We may assume that $\psi$ is normalized, i.e., $\psi(g,1) = \psi(1,g)=1$ for all $g \in G$. 
For any (normalized) map $\phi: G\rightarrow U$, the cocycle $\partial\phi$  
defined by 
$\partial\phi(g,h)=\phi(g)^{-1}\phi(h)^{-1}\phi(gh)$ is
a {\em coboundary}. The set of all cocycles 
$\mbox{$\psi:\abk G\times G\rightarrow U$}$ forms an 
abelian group $Z^2(G,U)$ under pointwise multiplication. 
Factoring out the subgroup of coboundaries gives $H^2(G,U)$,
the \emph{second cohomology group of $G$ with coefficients in $U$}.

Given a group $G$ and $\psi\in Z^2(G, U)$, 
denote by $E_\psi$ the canonical 
central extension of $U$ by 
$G$; this has elements $\{(u,g) \mid u\in U,\; g\in G \}$ and 
multiplication $(u,g) \;(v,h)=(uv\hspace{.5pt} \psi(g,h),gh)$.
The image $U\times \{1\}$ of $U$ lies in the centre of $E_\psi$ and the set $T(\psi)=\{(1,g)\,\colon\, g\in G\}$ is a normalized transversal of $U\times \{1\}$ in $E_\psi$.
In the other direction, suppose that $E$ is a finite group with 
normalized transversal $T$ for a central subgroup 
$U$. Put
$G= E/U$ and 
$\sigma (tU) = t$ for $t\in T$. The map
$\psi_T:G\times G\rightarrow U$ defined by
$\psi_T(g,h)=\sigma(g)\sigma(h)\sigma(gh)^{-1}$ is a cocycle; 
furthermore, $E_{\psi_{T}}\cong E$.

Each cocycle $\psi\in Z^2(G,U)$ is displayed as a 
{\em cocyclic matrix} $M_\psi$: under some indexing of 
the rows and 
columns by $G$, $M_\psi$ has entry $\psi(g,h)$ in 
position $(g,h)$.

A cocycle $\psi\in Z^2(G,U)$ is called {\em orthogonal} if, for each $g\neq 1\in G$ and each $u\in U$, $|\{h\in G \;\colon\; \psi(g,h)=u\}|=v/w.$ This definition arose as an equivalent formulation of the condition that the $G$-cocyclic matrix $M_\psi$ be a generalized Hadamard matrix $\GH(w,v/w)$ over $U$. We recall that a $v\times v$ matrix $H$ with entries in $U$, where $w$ divides $v$, is {\em a generalized Hadamard matrix $\GH(w,v/w)$} if, for every $i,j,\;1\leq i<j\leq v$, each of the multisets $\{h_{ik}h_{jk}^{-1} \mid 1\leq k\leq v\}$ contains every element of $U$ exactly $v/w$ times. A $\GH(w,v/w)$ is {\em normalized} if the first row and first column consist entirely of the identity element of $U$. We can always assume that our $\GH$ matrices are normalized.

Let $E$ be a group of order $vw$ with a normal subgroup $Z$ of order $w$. Suppose that $R$ is a $k$-subset of $E$, such that the multiset of quotients $r_1r_2^{-1}, \;r_i\in R,\;r_1 \neq r_2$, contains each element
of $E \setminus Z$ exactly $\lambda$ times, and contains no element of $Z$. Then $R$ is called {\em a $(v, m, k, \lambda)$-relative
difference set} in $E$ with forbidden subgroup $Z$. If $Z$ is a central subgroup of $E$ then we call
$R$ {\em a central relative difference set}.

For certain parameters, the existence of relative difference sets is equivalent to the existence of Hadamard matrices. The following result addresses this situation.

\begin{theorem}\label{ghequiv}\cite[Theorem~4.1]{PH98}
The following statements are equivalent.
\begin{enumerate}
\item $\psi\in Z^2(G,U)$ is orthogonal.
\item $M_\psi$ is a (normalized) $\GH(w,v/w)$.
\item There is a (central) relative $(v,w,v,v/w)$-difference set $T(\psi)=\{(1,g) : g\in G\}$ in the central extension $E_\psi$ of $U$ by $G$, relative to $U\times \{1\}$.
\end{enumerate}
\end{theorem}

%\marginpar{\tiny Ejemplos de matrices de Hadamard generalizadas y coc\'iclicas donde el c\'odigo asociado {\bf no sea lineal. $=\partial\phi_{(a,b)}$.}}
%\marginpar{\tiny Por el Corollary 28 del Articulo de Dougherty, Rifa y Merce; se deduce que rango de $C_H$ (códigos obtenidos del ejemplo 1) es mayor o igual a 2. Si Iv\'an probara que el rango del kernel de $C_H$ es 1. Tendr\'iamos  una demostración (alternativa al libro de Kathy p. 227)  de que todos estos c\'odigos son no lineales.}
\begin{example}\label{planar_exa}\cite[Example 9.2.1.4 and Theorem 9.48]{Hor07}
Let $G$ be the additive group of the finite field $\mathbb{F}_{3^a}$ and $\phi_{(a,b)}(g)=g^{(3^b+1)/2},\,\,g\in G$ where $(a,b)=1$, $b$ is odd and $1<b<2a-1$. Then
$$\partial\phi_{(a,b)}(g,h)=\phi_{(a,b)}(g+h)-\phi_{(a,b)}(g)-\phi_{(a,b)}(h)$$
is an orthogonal coboundary. Hence, $M_{\partial\phi_{(a,b)}}$ is a $\GH(3^a,1)$. Later, we will deal with $a=4$ and $b=3$.
\begin{remark}\label{propiedadesplanar} \begin{enumerate}
    \item Coulter and Mathews  found $\phi_{(a,b)}$ as a new class of planar power functions over $\mathbb{F}_{3^a}$ (see \cite{CM97}).
    \item  The symmetric orthogonal coboundaries $\partial\phi_{(a,b)}$ cannot be multiplicative. In particular, the resulting ternary Hadamard codes are not linear $3^a$-ary codes (see \cite[p.227]{Hor07}).
    \item The orthogonal coboundaries $\partial\phi_{(a,b)}$ and $\partial\phi_{(a,2a-b)}$ determine equivalent Hadamard codes (see \cite[Lemma 4.1]{HU03}). Hence we may restrict to the range $3\leq b\leq a-1.$
    \end{enumerate}

\end{remark}

\end{example}

Let $\mathbb{F}_{q}$ denote the finite field of order $q = p^r$, where $p$ is prime.  $\mathbb{F}_{q}$ is an additive elementary abelian group of order $q$.  From $H$ a normalized generalized Hadamard matrix $\GH(q,v/q)$ over $\mathbb{F}_{q}$, we denote by $F_H$ the $q$-ary code consisting of the rows of $H$, and $C_H$ the one defined as $C_H=\cup_{\alpha\in \mathbb{F}_{q}}(F_H+\alpha \bf{1})$ where ${\bf 1}$ denotes  the all-one vector (and $\alpha \bf{1}$ the all-$\alpha$ vector). The code $C_H$ over $\mathbb{F}_q$ is called {\em generalized Hadamard  code} which has $qv$ codewords, %\marginpar{\tiny Ivan: creo que son $qv$ codewords, length $v$, minimum distance $(q-1)v/q$ o $v-v/q$}, 
length $v$ and minimum distance $v-\frac{v}{q}$. Note that $F_H$ and $C_H$ are generally nonlinear codes over $\mathbb{F}_q$.

An ordinary Hadamard matrix of order $v=4t$ corresponds to a $\GH(2,2t$), where $U=\langle -1\rangle$. In this case two further equivalences are known.
\begin{proposition}\label{binaryresultadoprincipal}
When $U=\langle -1 \rangle \cong \mathbb{Z}_2$, the equivalent statements of Theorem \ref{ghequiv} are further equivalent to the following statements.
\begin{itemize}
\item[4.] There is a Hadamard group $E_\psi$ \cite{Fla97}.
\item[5.] $C_H$  is a Hadamard full propelinear code \cite{RS18}.%can be endowed with a full propelinear structure \cite{RS18}. 
\end{itemize}
\end{proposition}
Binary propelinear codes have been deeply studied in the literature, see \cite{ABBR18,BMRS12,PR02,RBH89} among other references. The Hadamard full propelinear codes are an important family of this type of codes.  However, aside from \cite{BMRS13} not much has been done for $q$-ary propelinear codes, especially for the class of full propelinear codes. 
In this paper, we prove the analog  of Proposition \ref{binaryresultadoprincipal}  when $U$ is 
%any abelian finite group. Mainly, we focus in equivalence {\em 5.} for $U$ being the additive group of $\mathbb{F}_q$. 
  the additive group of a finite field (i.e.  additive elementary abelian group). As a consequence, the class of generalized Hadamard full propelinear codes is introduced. Concerning equivalence {\em 4.}, let us mention that the Hadamard group $E_\psi$ in the binary case is effectively what is referred to as the {\em extension group} of a cocyclic Hadamard matrix, which is also defined for generalized Hadamard matrices with entries in $U$. Therefore, if the existence of a generalized Hadamard full propelinear code is equivalent to the existence of an orthogonal cocycle $\psi$, then there is an extension group $E_\psi$. Finally, let us point out that it seems that a generalized Hadamard matrix over any abelian group $U$ (should it exist) would afford the same theory, assuming similar definitions of propelinear codes over groups and so forth.

\section{Propelinear codes} 
%\marginpar{\tiny He usado en la demostración del Teorema 2 que $\pi_y^{-1}=\pi_{y^{-1}}$.}
Let $\mathbb{F}_q^n$ be the vector space of dimension $n$ over  $\mathbb{F}_q$.
The \emph{Hamming distance} between two vectors $v, w \in \mathbb{F}_q^n$, denoted by $d(v, w)$, is the number of the coordinates in which $v$ and $w$ differ.
A {\em($q$-ary) code} $C$ over $\mathbb{F}_q$ of length $n$ is a nonempty subset of $\mathbb{F}_q^n$. The elements of $C$ are called {\em codewords}. A code $C$ over $\mathbb{F}_q$ is called \emph{linear} if it is a linear space over $\mathbb{F}_q$. The \emph{minimum distance} of a code is the smallest Hamming distance between any pair of distinct codewords. Without loss of generality, we shall assume, unless stated otherwise, that the all-zero vector, denoted by ${\bf 0}$, is in $C$.

Two structural properties of (nonlinear) codes are the rank and dimension of the kernel.
The \emph{rank} of a  code $C$, $r=\rank(C)$, is the dimension of the linear span of $C$.
The \emph{kernel} of a  $q$-ary code, denoted by ${\cal K}(C)$, is defined as ${\cal K}(C) := \{x \in \mathbb{F}_q^n : C + \alpha\; x = C\,\,\mbox{for all}\,\alpha\in\mathbb{F}_q\}$. The $p$-kernel of $C$ is defined as ${\cal K}_p(C)=\{x\in \mathbb{F}_q^n : C+x=C\}$.
Note  that ${\cal K}(C)$ is a linear subspace and ${\cal K}_p(C)$ is $\mathbb{F}_p$-additive.
We will denote the dimension of the kernel of $C$ by $k=\ker(C)$.
These two parameters do not always give a full classification of codes, since two nonisomorphic codes could have the same rank and dimension of the kernel. In spite of that, they can help in classification, since if two codes have different rank or dimension of the kernel, they are nonisomorphic. A code is linear if and only if its rank and the dimension of its kernel are equal to the dimension of the code. In some sense, these two parameters give information about the linearity of a code.

%\marginpar{referencias}
Assuming the Hamming metric, any isometry of $\mathbb{F}_q^n$ is given by a coordinate permutation $\pi$ and $n$ permutations $\sigma_1,\ldots,\sigma_n$ of $\mathbb{F}_q$. We denote by $\Aut(\mathbb{F}_q^n)$ the group of all isometries of $\mathbb{F}_q^n$:
$$\Aut(\mathbb{F}_q^n)=\{(\sigma,\pi)\,\colon \,\sigma=(\sigma_1,\ldots,\sigma_n) \,\mbox{with}\, \sigma_i\in\,  \Sym \mathbb{F}_q,\,\,\,
\pi\in{\cal S}_n\}
$$
where $\Sym \mathbb{F}_q$ and ${\cal S}_n$ denote, respectively, the symmetric group of permutations on  $\mathbb{F}_q$ and on the set   $\{1,\ldots,n\}$. 

For any $\sigma=(\sigma_1,\ldots,\sigma_n)$ where $\sigma_i \in \Sym \mathbb{F}_q $,  $\pi \in{\cal S}_n$ and $v\in \mathbb{F}_q^n$, $v=(v_1,\ldots,v_n)$, we write  $\sigma(v)$ and $\pi(v)$ to denote 
$(\sigma_1(v_1),\ldots,\sigma_n(v_n))$  and
$(v_{\pi^{-1}(1)},\ldots,v_{\pi^{-1}(n)})$, respectively. 

The action of $(\sigma, \pi)$ is defined as
$$
(\sigma,\pi)(v)=\sigma(\pi(v)) \quad\mbox{for any} \,\,v\in \mathbb{F}_q^n,
$$
and the group operation in $\Aut(\mathbb{F}_q^n)$ is the composition
$$ 
(\sigma,\pi)\circ (\sigma',\pi')=((\sigma_1\circ \sigma'_{\pi^{-1}(1)},\ldots,\sigma_n\circ \sigma'_{\pi^{-1}(n)}),\pi\circ\pi') \; \mbox{for all}\; (\sigma,\pi), (\sigma',\pi')\in \Aut(\mathbb{F}_q^n).
$$
Here and throughout the entire paper, we use the convention 
$f\circ g (v)=f(g(v)),$ for $v\in\mathbb{F}_q^n$.
We denote by $\Aut(C)$ the group of all isometries of $\mathbb{F}_q^n$ fixing the code $C$ and we call it the {\em automorphism group} of the code $C$.

%%%%%%%%%%% AUT(M)
At this point, we introduce some basic background on automophism group of a matrix. 
Let $K$ be a multiplicative group isomorphic to the additive elementary abelian group $\mathbb{F}_{q}$, and let $\phi : \mathbb{F}_{q} \rightarrow K$ be an isomorphism. \emph{An automorphism of a matrix} $M$ with entries in a group $K$ is a pair of monomial matrices $(P,Q)$ with non-zero entries in $K$ such that $PMQ^{*} = M$, where $Q^{*}$ denotes the matrix obtained from the transpose of $Q$ by replacing each non-zero entry with it's inverse in $K$, and matrix multiplication is carried out over the group ring $\mathbb{Z} [K]$.  \emph{The automorphism group} $\Aut(M)$ of $M$ is the set of all such pairs of matrices, closed under the multiplication $(P,Q)(R,S) = (PR,QS)$.  \emph{The permutation automorphism group of} $M$ is the subgroup $\mathrm{PAut}(M) \subset \mathrm{Aut}(M)$ comprised of all pairs of permutation matrices in $\mathrm{Aut}(M)$.

\begin{lemma}\cite{DF11}
Let $M$ be a {$K$}-monomial matrix of order $n$. Then $M$ has
a unique factorization $D_MP_M$ where $D_M$ is a diagonal matrix and $P_M$ is a permutation matrix.
\end{lemma}

Here we will focus on GH-codes, $C_H$, where $H$ denotes a generalized Hadamard matrix of order $v$ with entries in the additive elementary abelian group $\mathbb{F}_{q}$ and write $\phi(H) = [\phi(h_{ij})]_{1 \leq i,j \leq v}$.

In what follows, we will make explicit the correspondence between the  elements of the automorphism group $\mathrm{Aut}(\phi(H))$ and  certain isometries of $C_H$ (elements of $\mathrm{Aut}(C_H)$).
Let $(M,N)\in \mathrm{Aut}(\phi(H))$, $x=[[D_M]_{1,1},\ldots,[D_M]_{v,v}]$ and $X$ be the $v\times v$ matrix such that each column is equal to $x^T$. It follows that 
$\phi(X+H)=D_M\phi(H)$. Likewise, if $Y$ is a $v \times v$ matrix over $\mathbb{F}_{q}$ such that each row is equal to $y = [[D_N]_{1,1},\ldots,[D_N]_{v,v}]$, then $\phi(H-Y) = \phi(H)D_{N}^{*}$.  So, 
$\phi(X + P_MHP_N^{T} - Y)=M\phi(H)N^{*}= \phi(H)$. Thus $X + P_MHP_N^{T} - Y=H$ and $(\sigma, \pi)\in\mathrm{Aut}(C_H)$ where
$\sigma_i(u)=u+[D_N]_{ii}$ and $\pi(1,\ldots,v)=(1,\ldots,v)P_N$. We will say that $(\sigma, \pi)$ is the isometry of $C_H$ associated to the automorphism $(M,N)$ of $\phi(H)$. Now, the following question arises naturally:  Given an isometry $(\sigma, \pi)$ of $C_H$, is it possible to define an automorphism $(M,N)$ of $\phi(H)$ associated to $(\sigma, \pi)$? We will answer this question affirmatively in a particular case in the next section.

%Let $X$ be a $n \times n$ matrix over $\mathbb{F}_{q}$ such that each column is identical, and let $x = [X_{1,1}, \ldots, X_{n,1}]$.  Then let $D_{x}$ be the diagonal matrix with diagonal $x$. It follows that $\phi(X+H) = D_{x}\phi(H)$. Likewise, if $Y$ is a $n \times n$ matrix over $\mathbb{F}_{q}$ such that each row is identical, and $y = [Y_{1,1},\ldots,Y_{1,n}]$, then $\phi(H-Y) = \phi(H)D_{y}^{*}$.  Suppose that there is some pair of permutation matrices $(P,Q)$ such that $X + PHQ^{*} - Y= H$.  Then the monomial matrices $M = D_{x}P$ and $N = D_{y}Q$ are such that $M\phi(H)N^{*} = \phi(H)$.

%%%%%%%%%%%%
%%% lll%%
\begin{definition}\cite{BMRS13}
A $q$-ary code $C$ of length $n$, ${\bf 0}\in C,$ is called {\em properlinear} if for any codeword $x\in C$ there exist 
$\pi_x\in {\cal S}_n$ and $\sigma_x=(\sigma_{x,1},\ldots,\sigma_{x,n})$ with $\sigma_{x,i}\in \Sym \mathbb{F}_q$ satisfying:
\begin{itemize}
    \item[(i)] for any $x\in C$ it holds $(\sigma_x,\pi_x)(C)=C$ and 
    $(\sigma_x,\pi_x)({\bf 0})=x$,
    \item[(ii)] if $y\in C$ and $z=(\sigma_x,\pi_x)(y)$, then 
    $(\sigma_z,\pi_z)=(\sigma_x,\pi_x)\circ(\sigma_y,\pi_y)$.
\end{itemize}
\end{definition}
A $q$-ary code is called {\em transitive} if the $\Aut(C)$ acts transitively on its codewords, i.e., the code satisfies the property (i) of the above definition. 

For all $x \in C$ and for all $y\in \mathbb{F}_q^n$, denote by $\star$ the binary operation such that $x\star y=(\sigma_x, \pi_x)(y)$. Then, $(C,\star)$ is a group, which is not abelian in general. This group structure is compatible with the Hamming distance, that is, such that $d(x\star u, x\star v)=d(u,v)$ where $u,v\in\mathbb{F}_q^n$.

The vector ${\bf 0}$ is always a codeword where $\pi_{\bf 0}=Id_n$ is the identity coordinate permutation and $\sigma_{{\bf 0},i }=Id_q$ is the identity permutation on $\mathbb{F}_q \,\,\forall i$.
Hence, ${\bf 0}$ is the identity element in $C$ and 
$\pi_{x^{-1}}=\pi_x^{-1}$ and $\sigma_{x^{-1},i}=\sigma^{-1}_{x,\pi_x(i)}$
for all $x\in C$ and for all $i\in\{1,\ldots,n\}$. We also say that $(C,\star)$ is \emph{a propelinear code}.

Clearly, the propelinear class is more general than the linear code class.

\begin{proposition}
Let $(C,\star)\subset \mathbb{F}_q^n$ be a group. $C$ is propelinear code if and only if the group $\Aut(C)$ (the isometries) contains a regular subgroup acting  transitively on $C$.
\end{proposition}
\begin{proof}
%\marginpar{\tiny Mirar la proposition III.1 del articulo ''On binary 1-perfect additive codes: some structural properties''. IEEE transctions of inform. theory vol 48 n 9, 2587--2592 (2002).}
Firstly, we assume $C$ is propelinear. Let $\rho_x:C\rightarrow C$ given by $\rho_x(v)=x \star v$. Let $x,y,z$ be any codewords in $C$, we have $\rho_x \rho_y(z)= \rho_x (y \star z)=x \star (y \star z) = (x \star y) \star z= \rho_{x\star y} (z)$. From \cite[Lemma 5]{BMRS13}, we have $d(x \star y, x \star z) = d(y,z)$, and so $d(\rho_x(y),\rho_x(z))=d(x\star y, x\star z)=d(y,z)$. Therefore, $G=\{\rho_x \mid x\in G\}$ is a subgroup of $\Aut(C)$, and $|G|=|C|$. Given $x,y \in C$, we take $z=y\star x^{-1}$, and so we have $\rho_z(x)=z \star x= y \star x^{-1} \star x=y$. Hence, $G$ acts transitively on $C$.

Conversely, we assume $\Aut(C)$ contains a regular subgroup $G$ acting transitively on $C$, so $|G|=|C|$. We call $\rho_x$ the element of $G$ such that $\rho_x({\bf 0})=x$. Note that $G \rightarrow C$ given by $\rho_x \rightarrow x$ is a bijection since $G$ is regular and acts transitively on $C$.
For $x \in C$, we define $(\sigma_x,\pi_x)(y)=\rho_x(y)$. Note that $(\sigma_x,\pi_x) \in \mathrm{Aut}(C)$ because $\rho_x$ is an isometry on $C$. We define $x \star y = (\sigma_x,\pi_x)(y)=\rho_x(y)$, where $x \in C$. Let us see that the operation $\star$ is propelinear, and so $C$ has a propelinear structure.
It is clear that $(\sigma_x,\pi_x)(C)=\rho_x(C)=C$, and $x\star {\bf 0}=\rho_x({\bf 0})=x$ for any $x\in C$.
As $G$ acts transitively on $C$, we have $\rho_x\rho_y=\rho_{x\star y}$ if and only if $\rho_x\rho_y({\bf 0})=\rho_{x\star y}({\bf 0})$.
Let $x,y \in C$, then $\rho_{x\star y}({\bf 0})=x \star y = \rho_x(y)=\rho_x(\rho_y({\bf 0}))=\rho_x \rho_y ({\bf 0})$. Thus, $(\sigma_{x\star y},\pi_{x \star y})(z)=\rho_{x\star y}(z)=\rho_x \rho_y(z)=\rho_x((\sigma_y,\pi_y)(z))=(\sigma_x,\pi_x)\circ(\sigma_y,\pi_y)(z)$.
%\hfill $\Box$
\end{proof}

In the binary case, when $q=2$, taking the usual addition on $\mathbb{F}_2$, the above definition is reduced to the following:

 A binary code $C$ of length $n$ is  \emph{propelinear} \cite{RBH89} if for each codeword $x \in C$ there exists $\pi_x\in {\cal S}_n$ satisfying the following conditions for all $y \in C$:
\begin{itemize}
\item[(i)] $x+\pi_x(y) \in C$.
\item[(ii)] $\pi_x\pi_y=\pi_{x+\pi_x(y)}$.
\end{itemize}
Furthermore, $C$ is called {\em full propelinear} \cite{RS18} if the permutation $\pi_x$ has not any fixed coordinate when $x\neq {\bf 0}$, $x\neq {\bf 1}$; and if ${\bf 1}\in C$ then $\pi_{\bf 1}=Id_n $.

\begin{definition}A \emph{full propelinear code} is a propelinear code $C$ such that for every $a \in C$, $\sigma_a(x)=a+x$ and $\pi_a$ has not any fixed coordinate when   $a\neq \alpha {\bf 1}$ for $\alpha\in \mathbb{F}_q$. Otherwise, $\pi_a=Id_n$.
\end{definition}

 A generalized Hadamard code, which is also full propelinear, is called \emph{generalized Hadamard full propelinear code} (briefly, $\GHFP$-code). In the binary case, we have the Hadamard full propelinear codes, they were introduced in \cite{RS18} and their equivalence with Hadamard groups was proven.
 
 \begin{lemma} Let $(C,\star)$ be a $\GHFP$-code and $a, b\in C$. If $a-b=\lambda{\bf 1}$ where $\lambda\in \mathbb{F}_q$ then $\pi_a=\pi_b$.
  \end{lemma}
\begin{proof}
We have 
$b\star\lambda{\bf 1}=b+\pi_b(\lambda{\bf 1})=b+\lambda{\bf 1}=a$ and $a\star\lambda{\bf 1}=\lambda{\bf 1}\star a$.
On the other hand,
$\pi_a(x)=a\star x - a=(b\star\lambda{\bf 1})\star x-(b+\lambda{\bf 1})=
(b\star x)\star\lambda {\bf 1}-(b+\lambda{\bf 1})=(b\star x)+\lambda {\bf 1}-(b+\lambda{\bf 1})=b\star x-b=\pi_b(x),$ for all $x\in C.$
%\hfill $\Box$
\end{proof}
\begin{lemma}
Let $C$ be a $\GHFP$-code  and $e_i$ be the unitary vector with only nonzero coordinate at the $i$-th position. If $x,y\in C$ then $\pi_x^{-1}(e_i)=\pi_y^{-1}(e_i)$ if and only if $x= y+\lambda{\bf 1}, \,\,\lambda\in \mathbb{F}_q$. Furthermore, if $x, y\in F_H$ then $x=y.$
\end{lemma}
\begin{proof}
We have $\pi_x^{-1}(e_i)=\pi_y^{-1}(e_i)\Leftrightarrow e_i=\pi_x\pi_y^{-1}(e_i)=\pi_x\pi_{y^{-1}}(e_i)=\pi_{x\star y^{-1}}(e_i)$. Since $C$ is full then $x\star y^{-1}=\lambda{\bf 1}$, $\,\lambda\in \mathbb{F}_q$.
%\hfill $\Box$
\end{proof}
\begin{lemma}
Let $C$ be a $\GHFP$-code, $\Pi=\{\pi_x\,\colon\, x\in C\}$ and $C_1=\{\lambda{\bf1}\colon \lambda\in \mathbb{F}_q\}$. Then $C_1\subset K(C)$ and 
$\Pi$ is isomorphic to $C/C_1$.
\end{lemma}
\begin{proof}
It is immediate that $C_1=\{\lambda{\bf1}\colon u\in \mathbb{F}_q\}$. 
The map $x\rightarrow \pi_x$ is a group homomorphism from $C$ to $\Pi$.  Since $C$ is full propelinear, the kernel of this homomorphism is $C_1$. Hence, we conclude with the desired result.
%\hfill $\Box$
\end{proof}
\section{GHFP-codes and cocyclic generalized Hadamard matrices }
From now on, $H$ denotes a generalized Hadamard matrix of order $v$ with entries in the additive elementary abelian group $\mathbb{F}_q$. $K$ denotes a multiplicative group isomporphic to the additive elementary abelian group $\mathbb{F}_q$, and let $\phi : \mathbb{F}_q\rightarrow K$ be an isomorphism.  Write $\phi(H) = [\phi(h_{ij})]_{1 \leq i,j \leq n}$. 

Consider the $qv \times v$ matrix $E_{\phi(H)}$ comprised of the $q$ blocks $k_{0}\phi(H),\ldots,$ $k_{q-1}\phi(H)$ where $K = \{1=k_{0},\ldots,k_{q-1}\}$. Assuming that $C_H$ is a $\GHFP$-code and $a,x\in C_H$, then the action of $a$ on $C_H$ defined by 
$$\rho_a(x)=a\star x=a+\pi_a(x)\in C_H,$$
 ($\rho_a \in \mathrm{Aut}(C_H)$) is equivalent to the action of $N^*$ on $E_{\phi(H)}$  by right matrix multiplication where $N^{*} = Q^{*}D_{-a}^{*}$, with $Q$ being the permutation matrix according to $\pi_{a}$, and $D_{a}$ the diagonal matrix with diagonal $ \phi(a)$. Since the action of $a$ on $C$ preserves $C$, there is a $qv \times qv$ permutation matrix $P'$ such that $P'E_{\phi(H)}N^{*} = E_{\phi(H)}$. Moreover, the rows of $E_{\phi(H)}$ are the rows of $\phi(H), k_{1}\phi(H),\ldots,k_{q-1}\phi(H)$.  Thus there is a $v \times v$ monomial matrix $M = D_{k}P$ with $k$ a vector of length $v$ over $K$ such that $M\phi(H)N^{*} = \phi(H)$, where for all $1 \leq i,j \leq v$ and $0 \leq d \leq q-1$, if $P'$ permutes row $j + dv$ to row $i$ then
\begin{itemize}
\item $P$ permutes row $j$ to row $i$, and
\item the $i$-th entry of $k$ is $k_{d}$.
\end{itemize}
Thus $(M,N)$ is an automorphism of $\phi(H)$, and if $a =\lambda {\bf 1}$ for some $\lambda \in \mathbb{F}_{q}$, then the corresponding automorphism is of the form $(\phi(-\lambda)I,\phi(-\lambda)I)$.  This proves the following.
\begin{theorem}\label{thm1}
If $H$ is a generalized Hadamard matrix over the additive abelian group of $\mathbb{F}_{q}$ such that the rows of $H$ comprise a $\GHFP$-code $C$, then the group $(C,\star) \cong R \subseteq \mathrm{Aut}(\phi(H))$.  Moreover, $(kI,kI) \in R$ for all $k \in K$, and $R$ acts transitively on rows of $\phi(H)$.
\end{theorem}

\begin{remark}
$R$  acts  transitively on rows of $\phi(H)$ since $\rho_a(x)=\rho_b(x)$ if and only if %\marginpar{\small Ivan: if and only if?} 
$a=b$ but not regularly since
$|R|\neq v$.
\end{remark}

%%%%%%%%%%%%%%%%%%%%%%%%
%\marginpar{\tiny No tenemos la REGULAR porque no hay unicidad. Puede haber m\'as de un elemento de R que mueva la fila i de $\phi(H)$ en la j (concretamente, los elementos de R que se correspondan con elementos de $a,b\in C$ tal que $a-b=\lambda{\bf 1}$. Para solucionar esto, lo que hace es usar el ''Huge Scenario'' $\mathcal{E}_{M}$.
%$\phi_{a+\lambda {\bf 1}}(x)=\phi_a(x)+\lambda {\bf 1}=\phi_a(x+\lambda {\bf 1}).$
%}

%%%%%%%%%%%%%%%%%%%%%%%%%%%%%%%
Now, for a generalized Hadamard matrix $M$ with entries in $K$, $\mathrm{Aut}(M) \cong \mathrm{PAut}(\mathcal{E}_{M})$ where $\mathcal{E}_{M} = [k_{i}k_{j}M]_{0\leq i,j\leq q-1}$ (this is a special case of \cite[Theorem 9.6.14]{DF11}).  Where $\Theta: \mathrm{Aut}(M) \rightarrow \mathrm{PAut}(\mathcal{E}_{M})$ is the isomorphism outlined in \cite[pp. 110--111]{DF11}), we note that the center of $\mathrm{Aut}(M)$ contains the group of pairs of diagonal matrices $Z = \{(kI,kI) : k \in K\}$, and thus $\Theta(Z)$ is a central subgroup of $\mathrm{PAut}(\mathcal{E}_{M})$. { We require that} %$\pi_{\lambda {\bf 1}} = 1_{\mathrm{Sym}(n)}$ 
$\pi_{\lambda {\bf 1}} = Id_n$
in order for $C$ to be full propelinear. The transitivity requirement of the group $(C,\star)$ on $C$ for full propelinear codes then gives the following.

\begin{theorem}
$C$ is a generalized Hadamard full propelinear code if and only if  there is a subgroup $R \subseteq \mathrm{Aut}(\phi(H))$ with $Z \subseteq R$ such that $ \mathrm{PAut}(\mathcal{E}_{\phi(H)})$ contains a regular subgroup $\Theta(R)$, with $\Theta(Z) \subseteq \Theta(R)$.
\end{theorem}
%{\color{red}
\begin{proof}
Let $K = \{1=k_{0},\ldots,k_{q-1}\}$ and $Z = \{z_{i} = (k_{i}I,k_{i}I) : i \in \{0,\ldots,q-1\}\}$ and let $C$ be a generalized Hadamard full propelinear code.  Theorem \ref{thm1} gives that $(C,\star) \cong R \subseteq \mathrm{Aut}(\phi(H))$ where $Z \subseteq R$, and $R$ acts transitively on the rows of $\phi(H)$.  Since $Z$ is central and acts only by multiplication on rows of $\phi(H)$, there is a right transversal $S$ of $Z$ in $R$ where for any $j \in \{1,\ldots,n\}$ there is $s_j \in S$ such that $\phi(H)_{j} = (s_{j}\phi(H))_{1}$.  Thus $\Theta(z_{i}s_{j})$ permutes row 1 of $\mathcal{E}_{\phi(H)}$ to row $iq + j$, proving that $\Theta(R)$ is transitive on rows of $\mathcal{E}_{\phi(H)}$.  By Theorem \ref{thm1}, $|R| = |(C,\star)|$ and thus $\Theta(R)$ acts regularly.

Conversely,  assuming  that $H$ is generalized Hadamard over $\mathbb{F}_{q}$ and that there is a subgroup $R \subseteq \mathrm{Aut}(\phi(H))$ with $Z \subseteq R$ such that $\Theta(R) \subseteq \mathrm{PAut}(\mathcal{E}_{\phi(H)})$ is regular and $\Theta(Z) \subseteq \Theta(R)$.  Label the rows of $\mathcal{E}_{\phi(H)}$ with the codewords of $C_{H}$ in the order of the rows of $E_{H}$ such that the first $n$ entries of the row of $\mathcal{E}_{\phi(H)}$ are the entries in the codeword labelling the row.  For any $x \in C_{H}$ there is $(M_{x}N_{x}) \in \Theta(R)$ such that $M_{x}$ sends row $x$ to row $0$.  In the preimage of $\Theta$, $N_{x}$ corresponds to a monomial matrix $D_{-x}Q_{x}$.  For each $x$, let $\pi_{x}$ be coordinate the permutation according to the action of $Q^{*}$ on columns of $\phi(H)$, and let $\sigma_{x}(a) = a+x$ for all $a \in E_{H}$, (i.e., $\pi_x$ and $\sigma_x$ are determined by the column action of $N_x$).  It follows that if $(\pi_{x},\sigma_{x})\circ (\pi_{y},\sigma_{y}) = (\pi_{z},\sigma_{z})$ then $N_{x}N_{y} = N_{z}$. It also follows that $(\sigma_{x},\pi_{x})(0) = x$ for all $x$.

Then let $f : \Theta(R) \rightarrow C_{H}$ be the map such that $f(M_{x},N_{x}) = x$.  Clearly this map is bijective.  Further, where %$u = (1,1,\cdots,1)$ and 
$\lambda \in \mathbb{F}_{q}$, it follows that $(M_{\lambda {\bf 1}},N_{\lambda {\bf 1}}) \in \Theta(Z)$, where $\pi_{\lambda {\bf 1}} = Id_n$.  Because $R \subseteq \mathrm{Aut}(\phi(H))$, it follows that $(\sigma_{x},\pi_{x})(C_{H}) = C_{H}$ for all $x$.  

%It remains to show that for any $x,y \in C_{H}$, if $x \star y = z$ then $(M_{x},N_{x})(M_{y},N_{y}) = (M_{z},N_{z})$, where $x \star y = (\sigma_{x},\pi_{x})(y)$, to show that $f$ is a homomorphism and $C_{H}$ has a propelinear structure.

Now observe that if $N_{x}N_{y} = N_{z}$, then $z = (\sigma_{z},\pi_{z})({\bf 0}) = (\sigma_{x},\pi_{x})(\sigma_{y},\pi_{y})({\bf 0}) = (\sigma_{x},\pi_{x})(y)$ and so $z = x \star y$.  Thus $f(M_{x},N_{x})\star f(M_{y},N_{y}) = x \star y = z = f(M_{z},N_{z})$, and so $f$ is a homomorphism and $C_{H}$ has a propelinear structure.
%\hfill $\Box$
\end{proof}%}

Let $G$ be a group of order $n$ and let $\psi: G \times G \rightarrow K$ be a 2-cocycle.  Then let $E_{\psi}$ denote the canonical central extension of $K$ by $G$ obtained from $\psi$.  The following is a special case of \cite[Theorem 14.6.4]{DF11}.

\begin{theorem}
A generalized Hadamard matrix $H$ over $K$ is cocyclic with cocycle $\psi$
if and only if there exists a centrally regular embedding of $E_{\psi}$ into $\mathrm{PAut}(\mathcal{E}_H)$.
\end{theorem}

\begin{corollary}\label{principalcorollary}
The code $C_{H}$ comprised of the rows of $E_{H}$ is a generalized Hadamard full propelinear code if and only if the matrix $H$ is cocyclic over some cocycle $\psi$, with extension group $E_{\psi} \cong R \cong (C_{H},\star)$ where $R$ is a regular subgroup of $\mathrm{PAut}(\mathcal{E}_H)$.
\end{corollary}

%===========================================

%\marginpar{\textcolor{blue}{\tiny Ahora vamos a estudiar de un modo expl\'icito la conexi\'on entre los cociclos y los GHFP-codes.}}

%===============================

%In this section we introduce the notion of  Generalized Hadamard full propelinear codes and their equivalence with central relative $(v,w,v,v/w)$-difference sets is studied. 

%\marginpar{\tiny I added this sentence here, it seemed appropriate.}
\begin{remark}
We observe that a generalized Hadamard matrix $H$ may be cocyclic over several distinct cocycles $\psi$, and that the extension groups $E_{\psi}$ are not necessarily isomorphic.  As such, given a cocyclic generalized Hadamard matrix $H$, there may be several codes $(C_{H},\star)$ that are equal setwise, i.e., they contain the same set of codewords, but are not isomorphic as groups.
\end{remark}

In what follows, we will make explicit the correspondence  between the elements of $E_\psi$ and $(C_{H},\star)$.

Assuming $\psi\in Z^2(G,K)$.
For a fixed order in $G=\{g_0=1,g_1,\ldots,g_{v-1}\}$ and in $K=\{k_0=1,k_1,\ldots,k_{q-1}\}$  (we recall that $K$ denotes the multiplicative group isomporphic to the additive elementary abelian group $\mathbb{F}_q$), we can define the following map:
$$
\Phi\colon E_\psi\rightarrow K^v
$$
given an element $(k,g)\in E_\psi$,
$$[\Phi(k,g)]_j=k_l, \quad \mbox{if $(k,g)^{-1}t_j\in T(\psi)\,(k_l,1)$,}
     $$
where $T(\psi)=\{(t_0=(1,1),t_2=(1,g_1),\ldots,t_{v-1}=(1,g_{v-1})\}$. Obviously,  $T(\psi)\,(c_i,1)=(c_i,1)T(\psi)$ and $\Phi$ is well-defined. After some calculations,
$$[\Phi(k,g)]_j=(k\psi(g,g^{-1}))^{-1} \,\psi(g^{-1},g_j).$$
 Hence, $\Phi(k,g)$ is equal to $(k\psi(g,g^{-1}))^{-1}$-times the row of $M_\psi$ indexed with the element $g^{-1}$.
 
  Clearly, $\Phi$ is an injective map. The inverse of $\Phi$ (over the Im$\,\Phi$) is
 $$
 \Phi^{-1}(\lambda\,(\psi(g,g_1),\ldots,\psi(g,g_v)))=
 ((\lambda \,\psi(g^{-1},g))^{-1},\,g^{-1}),
 $$
 where $\lambda\in K$ and $g\in G$.
 \begin{proposition} \label{pfromctocg}
If $\psi\in Z^2(G,K)$ is orthogonal  then $C=(\Phi(E_\psi),\star)$ is a $\GHFP$-code where
 $x\star y= \Phi(\Phi^{-1}(x)\cdot\Phi^{-1}(y)) $ with $x,y\in \Phi(E_\psi)$.
 \end{proposition}
 \begin{proof}
  Firstly, we will show that $\pi_x\in {\cal S}_v$ where
     $ \pi_x(y)=x\star y-x.$
 We know that every codeword has to be a multiple of a row of $M_\psi$. We take $x=\lambda\,(\psi(g,g_1),\ldots,\psi(g,g_v))$ and $y=\mu\,(\psi(h,g_1),\ldots,\psi(h,g_v))$.
 By a routine   computation, we get that 
 $$[x\star y]_j=\lambda\mu \psi(g^{-1},g)\psi(h^{-1},h)\left(\psi(g^{-1},h^{-1}) \psi((hg)^{-1},hg)\right)^{-1}\psi(hg,g_j).$$
 Putting together,
 $$\begin{array}{rcl}
     [\pi_x(y)]_j & = & [x\star y]_j-[x]_j  \\
     & =&  
     \mu \psi(g^{-1},g)\psi(h^{-1},h)\left(\psi(g^{-1},h^{-1}) \psi((hg)^{-1},hg)\right)^{-1} \psi(hg,g_j)
     (\psi(g,g_j))^{-1}\\[2mm]
     & = & \mu\psi(h,g g_j). 
   \end{array} $$ %\marginpar{\small Ivan: $gg_j$}
    In the last identity we have used these properties coming from (\ref{CocycleIdentity})
    \begin{itemize}
        \item $\psi(hg,g_j)
     (\psi(g,g_j))^{-1}=\psi(h,g g_j)(\psi(h,g))^{-1}.$%\marginpar{\small Ivan: $gg_j$}
     \item $\psi(h^{-1},h)(\psi(h,g))^{-1}=\psi(h^{-1},hg).$
     \item $\psi(g^{-1},h^{-1})\psi(g^{-1}h^{-1},hg)=\psi(g^{-1},g)\psi(h^{-1},hg).$
    \end{itemize}
 Hence, the map $\pi_x$ is an element of ${\cal S}_v$. Specifically, for any $y$, $\pi_x$ moves the $l$-th coordinate of $y$ to $j$-th coordinate where $g_l=g g_j$.  As a consequence of this fact, it is immediate that if $x=\lambda {\bf 1},\,\,\lambda\in\mathbb{F}_q$ then $\pi_x=Id_v$ since $g=1$ the identity of $G$. Furthermore, if $g\neq 1$ (or equivalently $x\neq\lambda {\bf 1}$), then $\pi_x$ has not any fixed coordinate.
 
 Secondly, we show an important property of these permutations. Concretely, given $x,y \in C$, we have that  $\pi_x\pi_y=\pi_{x\star y}$. To prove it, let $z$ be an element of $C$ then
 $$\begin{array}{lcl}\pi_{x\star y}(z) & = & (x\star y)\star z-x\star y \\
 & = & x\star (y\star z)-x\star y \\
 & = & x+\pi_x(y\star z)-\pi_x(y)-x\\
 & = & \pi_x(y\star z-y)
 =\pi_x(\pi_y(z)).
 \end{array}$$ 
 %\hfill $\Box$
\end{proof}
 
Let $H$ be a normalized Hadamard matrix $\GH(q,v/q)$ over $\mathbb{F}_q$ and $f$ be any row of $H$. $D_j$  denotes the subset of $C_H$ such that $x\in D_j$   %\marginpar{\small Ivan: iff? if?} 
if $[x]_j = 0 \in \mathbb{F}_{q}$. 
%\marginpar{\tiny should this be ``if $[x]_j = 0 \in \mathbb{F}_{q}$and $x \neq {\bf 0}$''? A lot of notation switching between $U$ and $\mathbb{F}_{q}$ in this part, I'm not sure which you want.
%\textcolor{blue}{Andr\'es: Corrected. Before ``$[x]_j$ is the identity element of $U$ and $x\neq {\bf 1}$''}}
Let us observe  the following facts:
 \begin{enumerate}
          \item $\displaystyle
     D_j=\cup_{\alpha\in \mathbb{F}_q} \{f+(-\alpha){\bf 1}\,\colon\,f\in F_H \;\wedge\; [f]_j=\alpha\}
     $.
     \item $D_1=F_H.$
     \item For $j>1$, $|\{f\in F_H\,\colon\, [f]_j=\alpha\}|=v/q$. %\textcolor{blue}{Idea:}\newline
     Since
     $\mathbb{F}_q$ is abelian then $H^T$ is a $\GH(q,v/q$) (over $\mathbb{F}_q$)  too \cite[Lemma 4.10]{Hor07}. Thus, the number of entries equal to $\alpha$ in the $j$-th column of $H$    is $v/q$, for all $\alpha\in \mathbb{F}_q$.
     \item  $|D_j|= v$ and $C=\cup_{i\geq 1} D_i$.
 \end{enumerate}

 %Let $e_i$ be the unitary
 \begin{proposition}\label{GHFP-SDR}
 Let $(C,\star)$ be a $\GHFP$-code of length $v$ over $\mathbb{F}_q$  coming from $H$ a $\GH(q,v/q)$. Then $F_H=D_1$ is a (central) relative  $(v,q,v,v/q)$-difference set in $C$ relative to the normal subgroup $C_1=\{\alpha{\bf1}\colon \alpha\in \mathbb{F}_q\}\cong \mathbb{F}_q$.
 \end{proposition}
 \begin{proof}

      $C_1$ is a central subgroup.
%    \marginpar{\tiny $x\star(e+u{\bf 1})=x+\pi_{x}(e+u{\bf 1})=e+u{\bf 1}+Id_v(x)=(e+u{\bf 1})\star x$, for all $x\in C.$}
     We have to prove:
     $$|F_H\cap x\star F_H|=\left\{\begin{array}{cl}
     v &\quad x={\bf 0}\\
     0 &\quad x\in C_1\setminus\{{\bf 0}\}\\
     v/q & \quad x\in C\setminus C_1
     \end{array}\right.$$
     \begin{itemize}
     \item Let us observe that if $x\in C_1$ then $\pi_x=Id_v$. Now, if $f\in F_H$ then the first entry of $x\star f=x+ f$ is $0$ if and only if $x={\bf 0}$. So, we concluded with the desired result for the first and the second identities. 
    % \marginpar{\tiny  first and second. justification: Let us observe that if $x\in C_1$ then $\pi_x=Id_v$. Now, if $f\in F_H$ then the first entry of $x\star f=x+ f=e+\alpha u+f$ is 0 if and only if $\alpha=0$, $\alpha\in \mathbb{F}_q$.}
     \item Let $x\notin C_1$ and $\pi_x(1)=j,\,$    ($j\neq 1$ since it is full propelinear).

     Let $\alpha_0\in \mathbb{F}_q$ be such that $[x+\alpha_0{\bf 1}]_j=0.$ Since $(x+\alpha_0{\bf 1})\star f\in D_j$ for all $f\in F_H$ and $|y\star F_H|=v$ for all $y\in C_H$ then $(x+\alpha_0{\bf 1})\star F_H=D_j$. As a consequence, 
$$x\star F_{H}=D_j-\alpha_0 {\bf 1}.$$     
    % . Taking into account that $|D_j|=v=|xF_H|$ (see below), $(x+\alpha{\bf 1})\star F_H=D_j$ and $(x+\alpha{\bf 1})\star f=x\star f+ \alpha{\bf 1}$, then 
     %$$(e+\alpha^{-1}{\bf 1})D_j=xF_{H}.$$

     Therefore,
     $|F_H\cap x\star F_H|=$ number of entries equal to $-\alpha_0$ in the $j$-th column of $H$ what it is equal to $v/q.$
     This conclude the proof.

     \end{itemize}
     
 %\end{enumerate}
  %\hfill $\Box$
 \end{proof}

 \begin{corollary}\label{corolariofundamental}
 Let $(C,\star)$ be a $\GHFP$-code of length $v$ over $\mathbb{F}_q$  coming from $H$ a $\GH(q,v/q)$. Let $G=C/C_1$ and $\sigma(f\star C_1)=f$ for $f\in F_H$. The map
 $\psi_{F_H}\colon G\times G \rightarrow K$ defined by
 $$\psi_{F_H}(g,h)=k,\quad \mbox{if}\,\sigma(g)\star\sigma(h)\in k{\bf 1}\star F_H $$
 is  an orthogonal cocycle,  i.e. $M_{\psi_{F_H}}$ is a $\GH(w,v/w$). Furthermore, $(C,\star)\cong E_{\psi_{F_H}}$ where $F_H^\star=\{(1,g)\colon g\in G\}$ is the isomorphic image of $F_H$.

 %Furthermore, $(C,\star)\cong E_{\psi_{F_H}}$ where $F_H^\star=\{(a_g,g)\colon g\in G\}$ is the isomorphic image of $F_H$ and the set mapping $\phi\,\colon\,G \rightarrow U$ given by $\phi(g)=(a_g)^{-1},\,\,g\in G$ is well defined.  Thus,
 %$$\psi=\psi_{F_H}\partial\phi \,\,\,\mbox{ is an orthogonal cocycle},$$
% i.e. $M_{\psi}$ is a GH($w,v/w$).
 
 \end{corollary}
 \begin{proof}
 It is a consequence of \cite[Theorem 3.1]{PH98}  and Proposition \ref{GHFP-SDR}.
 %\hfill $\Box$
 \end{proof}
  %\marginpar{\tiny \textcolor{blue}{Cuidado con el Theorem \ref{ghequiv} y el Corollay \ref{corolario1}. Mirar libro de Kathy p. 152 Corollary 7.31. Tal vez sea mejor escribir el Teorema \ref{ghequiv} as\'i. 
 %\begin{itemize}
     %\item $G=C/C_1$ es un grupo donde la ley de grupo se define as\'i:
    % $g\cdot h=g\star h \mod C_1=f_1$; esto quiere decir que $g\star h= u\cdot f_i$ con $f_i\in F_H$.
    % \item $E_{\psi_{F_h}}\cong C$ via de mapping $(u, f_i)=uf_i$
    % donde $E_{\psi_{F_h}}=C_1\times_{\psi_{F_h}} G$ y $C_1\cong U$.
     %\item $T(\psi_{F_h})=\{(1,g)\,\colon\, g\in G\}\cong F_H$ is a central relative difference set in $E_{\psi_{F_h}}$, relative to $U\times 1$.
 %\end{itemize}
% }
 %\textcolor{blue}{
 % $F_H^\star=\{(a_g,g)\colon g\in G\}$ tomamos 
 % $i(a_g)\star \sigma(g)=(a_g,\ldots,a_g)\star \sigma(g) \in (a_g,\ldots,a_g)\star F_H\cong F_H\Leftrightarrow a_g=a_h.$ Luego $F_H^\star=\{(1,g)\colon g\in G\},$ entonces $\psi=u\psi_{F_H}$ es ortogonal y por tanto $\psi_{F_H}$ tambien.
  %}}

\section{Examples}
In this section, we provide some examples of generalized Hadamard full propelinear codes coming from cocyclic generalized Hadamard matrices. The last one has a special interest since this family is not linear. We will study their rank and the dimension of their kernel. In \cite{DRV16} the study of the rank and dimension of the kernel of codes coming from generalized Hadamard matrices was initiated.  We begin with a definition of an infinite family of cocyclic generalized Hadamard matrices.%\marginpar{\tiny This bit and the following paragraph are new}

\begin{definition} {\upshape \cite[Section 9.2]{DF11}
Let $q = p^m$ be a prime power and denote the $k$-dimensional vector space over $\mathbb{F}_{q}$ by $V$. Then 
\[
D_{(p,m,k)} = [xy^{\top}]_{x,y \in V}
\]
is a $\GH(q,q^{k-1})$. These are known as the \emph{generalized Sylvester matrices}.}
\end{definition}

It is well known that the generalized Sylvester matrices are cocyclic, see \cite[p.122]{Hor07} for example. They were analyzed in terms of their cocyclic development in \cite{EF17}.  The analysis shows that these matrices have several non-isomorphic indexing and extension groups, and the number of non-isomorphic indexing and extension groups grows with $k$ and $m$.  They are closely related to the regular subgroups of the affine general linear group $\mathrm{AGL}_{k+1}(V)$.  Hence the matrix $H = D_{(p,m,k)}$ of order $q^k$ is cocyclic with multiple cocycles $\psi$ and has multiple non-isomorphic extension groups $E_{\psi}$ of order $q^{k+1}$.  As such, for each $\psi$ the associated codes $(C,\star)$ each have the same set of codewords (the rows of $E_{H}$), but are non-isomorphic as groups.  Some of the examples below are members of the generalized Sylvester matrices.

%\marginpar{\tiny $H(i,j)=0$ if $t_i^{-1}t_j\in T$\newline
%$Q(i,j)=0$ if $t_it_j\in T$\newline both of them define the same QH-code.}
\begin{example}
If $G = U = \langle a,b\mid
a^{2}=b^2=(ab)^2=1\rangle \cong {\bf Z}_2^2$ (the additive group of $\mathbb{F}_{4}$ but with multiplicative notation)  with indexing $\{1, a, b, ab\}$, then
the $G$-cocyclic matrix with coefficients in $U$
$$H=\left(\begin{array}{cccc}
1 & 1 & 1 & 1\\
1 & a & ab & b\\
1 & ab & b & a \\
1 & b & a & ab
\end{array} \right)
$$
is a generalized Hadamard matrix, $\GH(4,1)$, with entries in $\mathbb{F}_4$. 

\vspace{0.4cm}

Now, set $C_i=\{f_i+ \alpha {\bf 1}\mid \alpha \in G\}$, where $f_i$ denotes the vector corresponding to the $i$-th row of $H$ and ${\bf 1}$ denotes the all-one
 vector. (We will follow this notation in the sequel examples). For instance, 
 $$C_1=\{(1,1,1,1), (a,a,a,a), (b,b,b,b), (ab,ab,ab,ab)\}.$$

\vspace{0.4cm}

The generalized Hadamard  code over $U$
$$C=C_1\cup C_2\cup C_3\cup C_4$$
can be endowed with a {\bf full propelinear structure} with the following group $\Pi$ of permutations
$$\pi_x=\left\{
\begin{array}{cc}
I & \quad x\in C_1\\
(1,2)(3,4)& \quad x\in C_2\\
(1,3)(2,4)&\quad x\in C_3\\
(1,4)(2,3)& \quad x\in C_4
\end{array}
\right.$$
That is, $x\star y=x+\pi_x(y)$ where $(C,\star)\cong {\bf Z}_4^2$ and $\Pi\cong {\bf Z}_2^2$. The rank and the dimension of the kernel of this code are 2.
\end{example}

\begin{example}If $G = {\bf Z}_3^2$ with indexing $\{(0,0),(0,1),(0,2),(1,0),(1,1),(1,2),$ $(2,0),(2,1),(2,2)\}$, then
the $G$-cocyclic matrix over $\mathbf{Z}_3$
$$H=\left(\begin{array}{ccccccccc}
0 & 0 & 0 & 0 & 0 & 0 & 0 & 0 & 0\\
0 & 1 & 2 & 0 & 1 & 2 & 0 & 1 & 2 \\
0 & 2 & 1 & 0 & 2 & 1 & 0 & 2 & 1  \\
0 & 0 & 0 & 1 & 1 & 1 & 2 & 2 & 2    \\
0 & 1 & 2 & 1 & 2 & 0 & 2 & 0 & 1 \\
0 & 2 & 1 & 1 & 0 & 2 & 2 & 1 & 0 \\
0 & 0 & 0 & 2 & 2 & 2 & 1 & 1 & 1    \\
0 & 1 & 2 & 2 & 0 & 1 & 1 & 2 & 0 \\
0 & 2 & 1 & 2 & 1 & 0 & 1 & 0 & 2
\end{array} \right)
$$
is a generalized Hadamard matrix (of Sylvester type), $\GH(3,3)$, with entries in $\mathbf{F}_3$. 
The generalized Hadamard  code over $G$
$$C=C_1\cup C_2\cup\ldots\cup C_9$$ 
can be endowed with a {\bf full propelinear structure} with the following group $\Pi$ of permutations
$$\pi_x=\left\{
\begin{array}{cc}
I &\quad x\in C_1\\
(1,2,3)(4,5,6)(7,8,9) & \quad x\in C_2\\
(1,3,2)(4,6,5)(7,9,8) &\quad x\in C_3\\
(1,4,7)(2,5,8)(3,6,9) &\quad x\in C_4\\
(1,5,9)(2,6,7)(3,4,8) &\quad x\in C_5\\
(1,6,8)(2,4,9)(3,5,7) &\quad x\in C_6\\
(1,7,4)(2,8,5)(3,9,6) &\quad x\in C_7\\
(1,8,6)(2,9,4)(3,7,5) &\quad x\in C_8\\
(1,9,5)(2,7,6)(3,8,4) &\quad x\in C_9\\
\end{array}
\right.$$

We have $C\cong {\bf Z}_3^3$ and $\Pi\cong {\bf Z}_3^2$. The rank and the dimension of the kernel of this code are 3.
\end{example}

\begin{example}
Let  $G=U= {\bf Z}_2^3 $ be  with indexing $\{0,1, x, x^2, x^3,x^4,x^5,x^6\}$ where

\begin{center}
\begin{tabular}{c|cccccccc}
$+$ & 0 & 1 & $x$ & $x^2$ & $x^3$ & $x^4$ & $x^5$ & $x^6$\\\hline
$0$ & 0 & 1 & $x$ & $x^2$ & $x^3$ & $x^4$ & $x^5$ & $x^6$\\
$1$ &  & 0 & $x^3$ & $x^6$ & $x$ & $x^5$ & $x^4$ & $x^2$\\
$x$ &  &  & 0 & $x^4$ & $1$ & $x^2$ & $x^6$ & $x^5$\\
$x^2$ &  &  &  & 0 & $x^5$ & $x$ & $x^3$ & $1$\\
$x^3$ &  &  &  &  & 0 & $x^6$ & $x^2$ & $x^4$\\
$x^4$ &  &  &  &  &  & 0 & $1$ & $x^3$\\
$x^5$ &  &  &  &  &  &  & $0$ & $x$\\
$x^5$ &  &  &  &  &  &  &  & $0$
\end{tabular}
\end{center}
\noindent 
then the $G$-cocyclic matrix over $U$
$$H=\left(
\begin{array}{cccccccc}
0 & 0 & 0 & 0 & 0 & 0 & 0 & 0\\
0 & 1 & x & x^2 & x^3 & x^4 & x^5 & x^6\\
0 & x & x^2 & x^3 & x^4 & x^5 & x^6 & 1\\
0 & x^2 & x^3 & x^4 & x^5 & x^6 & 1 & x\\
0 & x^3 & x^4 & x^5 & x^6 & 1 & x & x^2\\
0 & x^4 & x^5 & x^6 & 1 & x & x^2 & x^3\\
0 & x^5 & x^6 & 1 & x & x^2 & x^3 & x^4\\
0 & x^6 & 1 & x & x^2 & x^3 & x^4 & x^5\\
\end{array}
\right)$$
is a generalized Hadamard matrix, $\GH(8,1)$, with entries in $\mathbb{F}_8$. The generalized Hadamard  code over $G$
$$C=C_1\cup C_2\cup \ldots \cup C_8$$
can be endowed with a {\bf full propelinear structure} with the following group $\Pi$ of permutations
$$\pi_x=\left\{
\begin{array}{cc}
I & \quad x\in C_1\\
(1,2)(3,5) (4,8) (6,7) &\quad x\in C_2\\
(1,3)(2,5) (4,6) (7,8) &\quad  x\in C_3\\
(1,4)(2,8) (3,6)(5,7)& \quad x\in C_4\\
(1,5)(2,3) (4,7)(6,8)&\quad x\in C_5\\
(1,6)(2,7) (3,4)(5,8)&\quad x\in C_6\\
(1,7)(2,6) (3,8)(4,5)&\quad x\in C_7\\
(1,8)(2,4) (3,7)(5,6)&\quad x\in C_8
\end{array}
\right.$$
We have  $(C,\star)\cong {\bf Z}_4^3$ and $\Pi\cong {\bf Z}_2^3$. The rank and the dimension of the kernel of this code are 2.

%\vspace{0.4cm}

%Now, set $C_i=\{f_i+ \alpha {\bf 1}\mid \alpha \in U\}$, where $f_i$ denotes the vector corresponding to the $i$-th row of $H$ and ${\bf 1}$ denotes the all-one vector. For instance, $$C_1=\{(0,0,0,0,0,0,0,0),(1,1,1,1,1,1,1,1),\ldots, (x^6,x^6,x^6,x^6,x^6,x^6,x^6,x^6) \}.$$

%\vspace{0.4cm}
\end{example}

\begin{example}
Let  $G=U= {\bf Z}_3^4 $ be  with indexing
$\{0000,0001,0002, 0010, \ldots,$ $ 2222\}$,  the irreducible polynomial which defines multiplication in the field is $2 + x + x^4$ and let $\phi_{(4,3)}$ as in Example \ref{planar_exa}. Then the $G$-cocyclic matrix over $U$
$$[H]_{g,h}=\partial\phi_{(4,3)}(g,h)
$$
is a generalized Hadamard matrix, $\GH(81,1)$, with entries in $\mathbb{F}_{81}$. 
$$C=C_1\cup C_2\cup \ldots \cup C_{81}$$
can be endowed with a {\bf full propelinear structure}. The group  $\Pi$ of permutations and the matrix $[H]_{g,h}$ can be downloaded from the following website (\href{https://ddd.uab.cat/record/204295?ln=en}{ddd.uab.cat/record/204295}). %website (\url{https://ddd.uab.cat/record/204295?ln=en}).
%(\textcolor{blue}{It is too long I think we should post it to a  webpage address}) .

%\marginpar{\tiny Mirar permutaciones.txt en articulos$\backslash$codigos}

We have that $(C,\star)\cong {\bf Z}_3^8$ and $\Pi\cong {\bf Z}_3^4$.  The rank of this code is 11 and the dimension of the kernel is 1. So, $C$ is not linear as we knew.
\end{example}

{In Table \ref{tabla}, we consider the codes associated to $\partial\phi_{(a,b)}$ of Example \ref{planar_exa}.  Let us recall that $\phi_{(a,b)}(g)=g^{(3^b+1)/2}$, with $g\in \mathbb{F}_{3^a}$. Moreover, if $(a,b)=1$, $b$ odd and $3\leq b\leq a-1$ then $\partial\phi_{(a,b)}$ are  orthogonal cocycles and the associated $\GHFP$-codes $C_{a,b}$ are not linear but are they inequivalent? that is, fixed $a$ and assuming that $b_1$ and $b_2$ with $b_1\neq b_2$ are admissible values,  are $C_{a,b_1}$ and $C_{a,b_2}$ inequivalent? If the conjecture below were true, we would have an affirmative answer. For instance, for $a=7$ we have two (cocyclic)  $\GH(3^7,1)$ matrices (one for $b=3$ and another for $b=5$) where their codes ($C_{7,3}$ and $C_{7,5}$) are inequivalent since they have different rank. Consequently, the $\GH$ matrices are nonequivalent as well.} %Although, to check whether two $\GH$ matrices are equivalent is known to be an NP-hard problem.  

%\begin{center}
%\begin{table}[h]\label{tabla}
%\[
%\begin{array}{|c||c|c|c|c|c|c|c|} \hline 
%\  b\backslash a  \ & 4 & 5 & 6 & 7 & 8 & 9 & 10\ \\
%\hline \hline
%3  &  (11,1)  &  (11,1) & & (11,1) &  (11,1) &    & (11,1)\\
%\hline 
%5  &   &  & (47,1) &    (47,1) &  (47,1) & (47,1) & \\
%\hline 
%7 & & & & & (191,1) & (191,1) & (191,1)
%\\
%\hline
%\hline 
%9 & & & & &  &  & (767,1) 
%\\
%\hline
%\end{array}
%\]
%\caption{The pairs $(r,k)$ of the entries of this table denote the rank and the dimension of the kernel of the  $\GHFP$-codes $C_{a,b}$. Let us notice that we have computed $(r,k)$ for all admissible value of $b$ for each $a$ in the range $3\leq b\leq a-1$ and $4\leq a\leq 10$. All these computations have been carried out with MAGMA.  }
%\end{table}
%\end{center}

\begin{table}[ht]
\caption{The pairs $(r,k)$ of the entries of this table denote the rank and the dimension of the kernel of the  $\GHFP$-codes $C_{a,b}$ associated to $\partial\phi_{(a,b)}$ of Example \ref{planar_exa}. }
\label{tabla}       % Give a unique label
\begin{tabular}{cccccccc}
\hline\noalign{\smallskip}
$ b\backslash a$ & 4 & 5 & 6 & 7 & 8 & 9 & 10  \\
\noalign{\smallskip}\hline\noalign{\smallskip}
$3$ & (11,1) & (11,1) &  & (11,1) & (11,1) &  & (11,1)\\
$ 5$ &  &  & (47,1) & (47,1) & (47,1) & (47,1) & \\
 $7$ &  & & &  & (191,1) & (191,1) & (191,1) \\
 $9$ & &  &  &  &  &  & (767,1)\\
\noalign{\smallskip}\hline
\end{tabular}
\bigskip

Let us notice that  in Table \ref{tabla}, we have computed the rank and dimension of the kernel for all admissible value of $b$ for each $a$ in the range $3\leq b\leq a-1$ and $4\leq a\leq 10$. All these computations have been carried out with \textsc{magma} \cite{magma}. We prove in Corollary \ref{ckerej1} that always $k=1$ and for the rank  we conjecture that $r$ depends only on $b$ by $r(b)=3\cdot 2^{b-1}-1$ with $b$ odd.
\end{table}

%\marginpar{\tiny Andr\'es $\Rightarrow$ Iv\'an. Podr\'ias echarle un vistazo a esta conjectura (est\'a en azul en la leyenda de la Tabla 1)? Tal vez se pueda demostrar estudiando el rango de $F_H$. Recuerda que el $\rank(C_H)=\rank(F_H)+1$.\\
%Ivan: $r(b)=3\cdot 2^{b-1}-1$}

\section{Kronecker sum construction}
In this section we extend the classical construction of Hadamard codes, based on Kronecker products, to the case of $\GHFP$-codes. As application, we construct an infinite family of nonlinear $\GHFP$-codes for each  $\GH(3^a,1$) matrix as in Example \ref{planar_exa}. Some properties of their rank and the dimension of their kernel are studied and they have been used to prove their nonlinearity.

The \emph{Kronecker sum construction} \cite{Shr64} is a standard method to construct $\GH$ matrices from other $\GH$ matrices. That is, if $H=(h_{i,j})$ is any $\GH(w,v/w$) matrix over $U$ and $B_1, \;B_2,\ldots,\; B_{v}$ are any $\GH(w,v'/w)$ matrices over $U$ then the matrix
$$
  H\oplus [B_1,B_2,\ldots,B_v]=\left(
\begin{array}{ccc}
    h_{11}+B_1 &\ldots & h_{1v}+B_1  \\
    \vdots & \vdots &  \vdots\\
     h_{v1}+B_n &\ldots & h_{vv}+B_v
\end{array}
\right)   
$$

is a $\GH(w,vv'/w)$ matrix. If $B_1=B_2=\ldots=B_v=B,$ then we denote $H\oplus [B_1,B_2,\ldots,B_v]$ by $H\oplus B.$

If $\psi\in Z^2(G,U)$ and $\psi'\in Z^2(G',U)$, then their \emph{tensor product} $\psi\otimes \psi'\in Z^2(G\times G',U)$, where
$$
(\psi\otimes \psi')((g,g'),(h,h'))=\psi(g,h)\psi(g',h'),
$$
and $M_{\psi\otimes \psi'}=M_\psi\oplus M_{\psi'}.$
%\marginpar{\tiny If $\psi\in Z^2(G,U)$ and $G$ abelian, the transpose $\psi^T\in Z^2(G,U) $  where 
%$\psi^T(g,h)=\psi(h,g)$. \textcolor{blue}{Relacion entre $C_H$ y $C_{H^T}$.}}

Let $S_q$ be the normalized $\GH(q,1$) matrix given by the multiplicative table of $\mathbb{F}_q$. We can recursively define $S^t$ as a $\GH(q,q^{t-1}$) matrix, constructed as $S^t=S_q\oplus S^{t-1}$ for $t>1$, (this is an alternative definition for the generalized Sylvester Hadamard matrices). %\marginpar{\tiny edited slightly because of the addition of the sylvester paragraphs earlier.}
It is well-known that $S_q$ is cocyclic (see \cite[p. 122]{Hor07}) and $\rank(C_{S_q})=\ker(C_{S_q})=2$. 
\begin{lemma}\cite[Lemma 3]{DRV16}
Let $H_1$ and $H_2$ be two $\GH$ matrices over $\mathbb{F}_q$ and $H=H_1\oplus H_2$. Then $\rank(C_H)=\rank(C_{H_1})+\rank(C_{H_2})-1$ and 
$\ker(C_H)= \ker(C_{H_1})+ \ker(C_{H_2})-1$.
\end{lemma}
Immediate consequences of the result above are that
$\rank(C_{ S^l})=\ker(C_{ S^l})=l+1$. On the other hand, 
if $H_1$ is linear and $H_2$ is not (or vice versa) then $H=H_1\oplus H_2$ is not linear.
\begin{lemma}\cite[Corollary 28]{DRV16}\label{c28drv16}
Let $H$ be a $\GH(q,q^{h-1}$) matrix over $\mathbb{F}_q$, with $q>3$ and $h\geq 1$, or $q=3$ and $h\geq 2$. Then $\rank(C_{H})\in \{h+1,\ldots, \lfloor q^h/2\rfloor\}$.
\end{lemma}
\begin{lemma}\cite[Proposition 9]{DRV16}\label{p9drv16} Let $H$ be a $\GH(q,\lambda$) over $\mathbb{F}_q$, where $q=p^e$ and $p$ prime. Let $v=q\lambda=p^t s$ such that $\gcd(p,s)=1$. Then $1\leq \ker(C_H)\leq \ker_p(C_H)\leq 1+t/e$.
\end{lemma}
%The following lemma is a generalized version of \cite[Lemma 5.1]{mog}.
% \begin{lemma}
% Let $C$ be a generalized propelinear code. Then $\mathcal{K}(C)$ is a subgroup of $C$.% Moreover, $x * \mathcal{K(C)} = x + \mathcal{K(C)}$ and $\pi_x(\mathcal{K(C)}) = \mathcal{K(C)}$, for all $x \in C$.
% \end{lemma}
% \begin{proof}
% As ${\bf 0} \in C$, we have that $\mathcal{K}(C)$ is linear. Let $x,y$ be in $\mathcal{K}(C)$, so $\alpha x + C = C$ and $\alpha y + C = C$ for all $\alpha \in \mathbb{F}_q$.
% Therefore, $\alpha (x \star y) + C = \alpha (\sigma_x,\pi_x)(y)+ C=^* 
% $\marginpar{\tiny $^*$Este paso creo que no es correcto. Si no lo soluciono lo mejor ser\'a eliminar el resultado ya que no se utiliza en lo siguiente.

% \textcolor{blue}{ Andr\'es:} Creo que es mejor quitar el Lemma 8}$
% (\sigma_x,\pi_x)(\alpha y)+(\sigma_x,\pi_x)(C)=(\sigma_x,\pi_x)(\alpha y+C)=(\sigma_x,\pi_x)(C)=C$, and so $x\star y \in \mathcal{K}(C)$. Thus, the operation $\star$ is closed on $\mathcal{K}(C)$. Since $\mathcal{K}(C)$ is finite and ${\bf 0} \in C$, we have that $\mathcal{K}(C)$ is a subgroup.
% \hfill $\Box$
% \end{proof}

\begin{lemma}\label{Ksubgroup}
Let $C$ be a generalized full propelinear code. Then $\mathcal{K}(C)$ is a subgroup of $C$.% Moreover, $x * \mathcal{K(C)} = x + \mathcal{K(C)}$ and $\pi_x(\mathcal{K(C)}) = \mathcal{K(C)}$, for all $x \in C$.
\end{lemma} 
% \marginpar{\tiny El enunciado incluye "full" para arreglar el error en el lemma anterior}
\begin{proof}
As ${\bf 0} \in C$, we have that $\mathcal{K}(C)$ is linear. Let $x,y$ be in $\mathcal{K}(C)$, so $\alpha x + C = C$ and $\alpha y + C = C$ for all $\alpha \in \mathbb{F}_q$.
Therefore, $\alpha (x \star y) + C = \alpha (x+\pi_x(y))+ x\star C= \alpha x + \alpha \pi_x(y)+x+\pi_x(C)= \alpha x + x + \pi_x(\alpha y+ C)= \alpha x + x + \pi_x(C)= \alpha x + x \star C= \alpha x + C = C$, and so $x\star y \in \mathcal{K}(C)$. Thus, the operation $\star$ is closed on $\mathcal{K}(C)$. Since $\mathcal{K}(C)$ is finite and ${\bf 0} \in C$, we have that $\mathcal{K}(C)$ is a subgroup.
%\hfill $\Box$
\end{proof}

\begin{proposition}Let $H$ be a $\GH(3^a,1$) over $\mathbb{F}_{3^a}$ where $C_H$ is a $\GHFP$-code. Then $\ker(C_H)\in \{1,2\}$. If $\ker(C_H)=2$, then $C_H$ is linear. Furthermore, if $a>1$, then $\rank(C_H)\geq 2$. 
%\marginpar{\small pendiente de ver que k=1 para el ejemplo 1.}
\end{proposition}
\begin{proof}
From Lemma \ref{p9drv16}, we have that $\ker(C_H) \in \{1, 2\}$.
We suppose that $\mathcal{K}(C_H)=\langle \mathbf{1}, x \rangle$, for some $x \in C_H$ with $x\neq \alpha \mathbf{1}$ for any $\alpha \in \mathbb{F}_{3^a}$. As the kernel is a linear subspace of $C_H$, we have that $\mathcal{K}(C_H)=\{ \alpha \mathbf{1} + \beta x : \alpha,\beta \in \mathbb{F}_{3^a} \}$. Thus, $|\mathcal{K}(C_H)|=3^{2a}=|C_H|$. Therefore $C_H=\mathcal{K}(C_H)$ and so $C_H$ is linear.

From Lemma \ref{c28drv16}, we have that $\rank(C_H)\geq 2$ if $a>1$.
%\hfill $\Box$
\end{proof}
\begin{corollary}\label{ckerej1}
Let $H=M_{\partial\phi_{(a,b)}}$ be as in Example \ref{planar_exa} %\textcolor{red}{then}
then
$\ker(C_H)=1$.
%\marginpar{\tiny $C_H$ is a nonlinear GHFP-code by Remark \ref{propiedadesplanar}}
\end{corollary}
\begin{proof}
$C_H$ is a nonlinear $\GHFP$-code by Remark \ref{propiedadesplanar}.
%\hfill $\Box$
\end{proof}

\begin{corollary}
If $q=3^a$ with $a>1$, $H$ a $\GH$ matrix over $\mathbb{F}_q$ where $C_H$ is a nonlinear $\GHFP$-code and $H'=S_q\oplus H$. Then $\rank(C_{H'})=\rank(C_H)+1>\ker(C_{H'})=2$.
\end{corollary}

\begin{proposition}\label{ptcocyclos}\cite[Theorem 6.9]{Hor07}
Let $\psi_i\in Z^2(G_i,U), \, 1\leq i\leq n$ and $\psi=\psi_1\otimes \cdots\otimes \psi_n\in Z^2(G_1\times\cdots\times G_n,U)$. Then $\psi$ is orthogonal if and only if $\psi_i$ is orthogonal, $1\leq i\leq n$.
\end{proposition}
\begin{remark}
As a direct consequence of Proposition \ref{ptcocyclos}, the Sylvester generalized Hadamard matrix $S^l$ is cocyclic. 
\end{remark}

%\begin{proposition}
%\emph{ Let $H\oplus [B_1,B_2,\ldots,B_v]$ be as in (\ref{dfnkc}). If $C_H,\,C_{B_1},\ldots,C_{B_v}$ are GHFP-codes then $C_{H\oplus [B_1,B_2,\ldots,B_v]}$ is a GHFP-code too.
%
%}
%\end{proposition}
%\marginpar{\tiny \textcolor{blue}{Mirar la tesis de suarez-canedo p. 83, proposition 143.}}
\begin{proposition}
Let $B_1$ be a  $\GH(w,v/w)$ matrix over $U$ and $B_2$ be a $\GH(w,v'/w)$ matrix over $U$. If $C_{B_1}$ and $C_{B_2}$ are $\GHFP$-codes then $C_H$ is a $\GHFP$-code too where $H=B_1\oplus B_2$.
Moreover,
$$\begin{array}{c}
   \pi_{a\oplus b}(x\oplus y)=\pi_a(x)\oplus \pi_b(y),    \\[2mm]
     (a\oplus b)\star (x\oplus y)= (a\star x)\oplus (b\star y).
\end{array}$$
where $a=(a_1,a_2,\ldots,a_v)$, $b=(b_1,b_2,\ldots,b_{v'})$  and $a\oplus b=(a_1+b_1,\ldots,a_1+b_{v'},a_2+b_1,\ldots,a_2+b_{v'},\ldots, a_v+b_1,\ldots, a_v+b_{v'})$ are rows in $B_1$, $B_2$ and $H$, respectively; $x\in C_{B_1}$ and $y\in C_{B_2}$.
\end{proposition}

\begin{proof}
By Corollary \ref{corolariofundamental}, we have that $B_i= M_{\psi_i}$ for $\psi_i\in Z^2(G_i,U)$ for a specific ordering of the elements of $G_i$ (for the rest of this proof, we are assuming fixed this ordering in $G_i$) with $i=1,2$. Now, using Proposition \ref{ptcocyclos}, we have $H=M_{\psi_1}\oplus M_{\psi_2}=M_{\psi_1\otimes \psi_2}$ for $\psi_1\otimes \psi_2\in Z^2(G_1\otimes G_2,U)$ which is orthogonal, i.e., $H$ is a cocyclic $\GH(w,vv'/w$). 
Therefore, by Proposition  \ref{pfromctocg}, $C_H$ is a $\GHFP$-code.

Now, assume that $a$ (resp. $b$) corresponds with a row of $B_1$ (resp. $B_2$) indexed with the element $g\in G_1$ (resp. $h\in G_2$). By the proof of Proposition \ref{pfromctocg}, we have
$\pi_a(l)=i\Leftrightarrow g_l=g g_i$ and $\pi_b(m)=j \Leftrightarrow h_m=h h_j$, where $g_j\in G_1$ and $h_j\in G_2$. For the same reason, 
$\pi_{a\oplus b}((l-1)v+m)= (i-1)v+j \Leftrightarrow(g_l,h_m)=(g,h)(g_i,h_j)$.
Therefore,
$\pi_{a\oplus b}(x\oplus y)=\pi_a(x)\oplus \pi_b(y).$ Finally, as a direct consequence, we conclude with the desired result
$(a\oplus b)\star (x\oplus y)= (a\star x)\oplus (b\star y).$
%\marginpar{\tiny 
%\begin{itemize}
 %   \item $\pi_{a\oplus b}\in{\cal S}_{v+v'}.$
  %  \item $\pi_{(a\oplus b)(c\oplus d)}=\pi_{a\oplus b} \pi_{c\oplus d}.$
   % \item $\pi_{a\oplus b}$ has no fixed points but $a\oplus b\in C_1$ where $\pi_{a\oplus b}=Id_{vv'}$.
    %\item \textcolor{blue}{ver que sucede cuando $B_1=M_{\psi_1}$ and $B_2=M_{\psi_2}$.Logicamente, $H=M_\psi$ donde $\psi=\psi_1\otimes \psi_2$.}
%\end{itemize}}
%\hfill $\Box$
\end{proof}
%\marginpar{\tiny los cociclos ortogonales se comportan bien con respecto al producto Kroneker. Theorem 6.9 p. 121, libro Kathy. Sería interesantes multiplicar un cociclo (con codigo no lineal) con otro cociclo (codigo lineal) y ver que pasa. Lo deseable sería que fuera no lineal. En cualquier caso, GHFP-codes se comportan bien con el producto Kronecker. Articulo de Dougherty, Rifa y Merce.}

\begin{corollary}
Let $\partial\phi_{(a,b)}$ be as in Example \ref{planar_exa} then $C_H$ are not linear $\GHFP$-codes  where $H={ S}^l\oplus M_{\partial\phi_{(a,b)}}$, for $l\geq 1$, are $\GH(3^a,3^{al}$) matrices  with ${ S}={ S}_{3^a}$. Moreover, $\ker(H)=l+1<\rank(H).$
\end{corollary}

\section*{Acknowledgements}
The authors would also like to thank  Kristeen Cheng for her reading of this manuscript.
The first author was supported by 
the project FQM-016 funded by JJAA (Spain). The second author was supported by the Catalan grant Borsa Ferran Sunyer i Balaguer and the Spanish grant TIN2016-77918-P (AEI/FEDER, UE). The third author was supported by the Irish Research Council (Government of Ireland Postdoctoral Fellowship, GOIPD/2018/304).

\newpage


\begin{thebibliography}{4}




%\bibitem{jour} Smith, T.F., Waterman, M.S.: Identification of Common Molecular
%Subsequences. J. Mol. Biol. 147, 195--197 (1981)

%\bibitem{lncschap} May, P., Ehrlich, H.C., Steinke, T.: ZIB Structure Prediction Pipeline:
%Composing a Complex Biological Workflow through Web Services. In: Nagel,
%W.E., Walter, W.V., Lehner, W. (eds.) Euro-Par 2006. LNCS, vol. 4128,
%pp. 1148--1158. Springer, Heidelberg (2006)

%\bibitem{book} Foster, I., Kesselman, C.: The Grid: Blueprint for a New Computing
%Infrastructure. Morgan Kaufmann, San Francisco (1999)

%\bibitem{proceeding1} Czajkowski, K., Fitzgerald, S., Foster, I., Kesselman, C.: Grid
%Information Services for Distributed Resource Sharing. In: 10th IEEE
%International Symposium on High Performance Distributed Computing, pp.
%181--184. IEEE Press, New York (2001)

%\bibitem{proceeding2} Foster, I., Kesselman, C., Nick, J., Tuecke, S.: The Physiology of the
%Grid: an Open Grid Services Architecture for Distributed Systems
%Integration. Technical report, Global Grid Forum (2002)

%\bibitem{url} National Center for Biotechnology Information, \url{http://www.ncbi.nlm.nih.gov}
%\bibitem{AAFG12a} \'Alvarez, V., Armario, J.A., Frau, M.D., and Gudiel, F.: The maximal determinant of cocyclic $(-1,1)$-matrices over $D_{2t}$. Linear Algebra Appl. 436,  858--873 (2012)  

%\bibitem{AAFG15} \'Alvarez, V., Armario, J.A., Frau, M.D., Gudiel, F.: Determinants of $(-1,1)$-matrices of the skew-symmetric type: a cocyclic approach. Open Math.  13, 16--25 (2015)

%\bibitem{AF17} Armario, J.A., Flannery D.L.: On quasi-orthogonal cocycles. J. Combin. Des., 2017; 00:1–-11. https://doi.org/10.1002/jcd.21597

\bibitem{ABBR18} Armario, J.A., Bailera, I., Borges, J., Rif\`{a}, J.: Quasi-Hadamard Full Propelinear Codes. Math. Comput. Sci. \textbf{12}, 419--428 (2018).

\bibitem{BMRS12} Borges, J., Mogilnykh, I.Y., Rif\`{a}, J., Solov'eva, F.I.:
\newblock Structural properties of binary propelinear codes.
\newblock Adv. Math. Commun. \textbf{6}(3),329--346 (2012).


\bibitem{BMRS13} Borges, J., Mogilnykh, I.Y., Rif\`{a}, J. Solov'{e}va, F.: On the number of nonequivalent propelinear extended perfect codes. Electronic J. Combinatorics \textbf{20}, 1--14 (2013).

\bibitem{magma} Bosma, W., Cannon, J.J., Fieker, C., Steel, A.:
\newblock Handbook of Magma functions, Edition 2.22 (2016). 


\bibitem{CM97} Coulter, R., Matthews, R.: Planar functions and planes of Lenz-Barlotti class II. Des. Codes and Cryptogr. \textbf{10}, 167--184 (1997).

%\bibitem{DFH00} de Launey, W., Flannery D.~L., and Horadam, K.~J.: Cocyclic Hadamard matrices and difference sets. Discrete Appl. Math. 102, 47--62 (2000)
\bibitem{DF11} de Launey, W., Flannery, D.L.: 
Algebraic design theory. 
Mathematical Surveys and Monographs, vol.~175. 
American Mathematical Society, 
Providence, RI (2011).


\bibitem{DRV16} Dougherty, S.T., Rif\`a, J., Villanueva, M.: Ranks and kernels of codes from generalized Hadamard matrices. IEEE trans. Inf. Theory \textbf{62}, 687--694 (2016).
%\bibitem{Ehl64} Ehlich, H.: Determiantenabsch\"{a}tzungen f\"{u}r bin\"{a}re Matrizen,  Math Z    { 83}, 123--132 (1964)

\bibitem{EF17}
Egan, R., Flannery, D.L.: Automorphisms of generalized Sylvester Hadamard matrices, Discrete Math. \textbf{340}(3), 516--523 (2017).

\bibitem{Fla97} Flannery, D.L.: Cocyclic Hadamard matrices and Hadamard groups are equivalent. J. Algebra \textbf{192}, 749--779 (1997).

%\bibitem{FKS04}Fletcher R. J.,  Koukouvinos C.,  Seberry J.: New skew-Hadamard matrices of order $4\cdot 59$ and new $D$-optimal designs of order  $2\cdot 59$, Discrete Math.   \textbf{286}, 252--253 (2004)


%\bibitem{Had93} Hadamard. J.:\newblock{ R\'esolution d'une question relative aux d\'eterminants,}\newblock {Bull Sciences Math} (2) { 17}, 240--246  (1893)
\bibitem{HU03} Horadam, K.J., Udaya, P.: A new class of ternary cocyclic Hadamard codes. Applic. Algebra En. Commun. Comp. \textbf{14}, 65--73 (2003).
\bibitem{Hor07} Horadam, K.J.: 
Hadamard Matrices and Their Applications. 
Princeton University Press, Princeton, NJ (2007).

%\bibitem{Ito94} Ito, N.:
%\newblock On Hadamard groups. 
%\newblock J. Algebra 168, 981--987 (1994)

%\bibitem{Kar85} Karpilovsky, G.: Projective representations of finite groups. Marcel Dekker, New York (1985) 

%\bibitem{KO06}  Kharaghani  H.,   Orrick. W.: D-optimal matrices, In: The CRC Handbook of Combinatorial Designs, C. J. Colbourn and J. Dinitz (Editors), CRC press, 2nd ed., Boca Raton, 2006 pp. 296--298.


%\bibitem{OS05}  Orrick, W., Solomon, O.: The Hadamard Maximal Determinant Problem (website), http://www.indiana.edu/$\sim\:$maxdet/, accessed 3 October 2017

\bibitem{PH98} Perera, A.A.I., Horadam, K.J.:
\newblock Cocyclic generalized Hadamard matrices and central relative difference sets.
\newblock Des. Codes Cryptogr. \textbf{15}, 187--200 (1998).

\bibitem{PR02} Phelps, K. T., Rif\`a, J.:
On binary 1-perfect additive codes: some structural properties. IEEE Trans. Inf. Theory \textbf{8}, 2587--2592 (2002).

\bibitem{RBH89} Rif\`a, J., Basart, J.M., Huguet, L.:
On completely regular propelinear codes.
In: Applied Algebra, Algebraic Algorithms and Error-Correcting Codes. LNCS 357,
pp. 341--355. Springer, Berlin (1989).

%\bibitem {RS14}  Rif\`a, J., Su\'arez, E.: About a class of Hadamard propelinear codes. Electronic Notes in Discrete Mathematics 46, 289--296 (2014)

\bibitem{RS18}  Rif\`a, J., Su\'arez, E.: Hadamard full propelinear codes of type $Q$. Rank and kernel. Des. Codes Cryptogr. \textbf{86}, 1905–-1921 (2018).

\bibitem{Shr64} Shrikhande, S.S.: Generalized Hadamard matrices and orthogonal arrays of strength two. Can. J. Math. \textbf{16}, 736--740 (1964).

%\bibitem{Woj64} Wojtas, W.:  On Hadamard's inequallity for the determinants of order non-divisible by 4. Colloq. Math.    12,  73--83 (1964)


\end{thebibliography}

\end{document}